\documentclass[12pt]{article}
\usepackage[a4paper,margin=2.5cm, centering]{geometry}
\usepackage{amsthm}
\usepackage{latexsym}
\usepackage{graphicx}
\usepackage{amsmath}
\usepackage{amssymb}
\usepackage{graphicx}
\usepackage[mathcal]{eucal} 
\usepackage{mathrsfs}
\usepackage{tikz}
\usepackage{bbm}
\usepackage{url}
\usepackage{subdepth}
\usepackage{hyperref}

\usepackage{ifpdf}
\ifpdf
\usepackage[all,pdf,cmtip]{xy} 
\else
\input xy
\xyoption{all}
\xyoption{v2}
\xyoption{cmtip}
\fi

\theoremstyle{plain}
\newtheorem{theorem}{Theorem}
\newtheorem{lemma}[theorem]{Lemma}
\newtheorem{proposition}[theorem]{Proposition}

\theoremstyle{definition}

\newtheorem{definition}[theorem]{Definition}

\newcommand{\cM}{\mathcal{M}}
\newcommand{\cX}{\mathcal{X}}
\newcommand{\cY}{\mathcal{Y}}

\newcommand{\vC}{\check{C}}
\newcommand{\cech}[1]{\vC(#1)}
\newcommand{\interior}[1]{%
{\kern0pt#1}^{\mathrm{o}}%
}

\DeclareMathOperator{\cHom}{\underline{\mathcal{H}\hspace{-0.1ex}om}}
\DeclareMathOperator{\Map}{Map}
\newcommand{\inhom}[2]{{#2}^{#1}}

\DeclareMathOperator{\Stack}{\bf{Stack}}

\begin{document}

\title{The smooth Hom-stack of an orbifold}

\author{David Michael Roberts
and Raymond F. Vozzo\footnote{This research was supported under the Australian Research Council's Discovery Projects funding scheme (project numbers DP120100106 and DP130102578). The authors thank Alexander Schmeding and Eugene Lerman for helpful discussion and suggestions. This document is released under a CC0 license \mbox{\url{http://creativecommons.org/publicdomain/zero/1.0/}}}}
\maketitle

\begin{abstract}

For a compact manifold $M$ and a differentiable stack $\cX$ presented by a Lie groupoid $X$, we show the Hom-stack $\cHom(M,\cX)$ is presented by a Fr\'echet--Lie groupoid $\Map(M,X)$ and so is an infinite-dimensional differentiable stack. We further show that if $\cX$ is an orbifold, presented by a proper \'etale Lie groupoid, then $\Map(M,X)$ is proper \'etale and so presents an infinite-dimensional orbifold.


\end{abstract}

This note serves to announce a generalisation of the authors' work \cite{Roberts-Vozzo}, which showed that the smooth loop stack of a differentiable stack is an infinite-dimensional differentiable stack, to more general mapping stacks where the source stack is a compact manifold (or more generally a compact manifold with corners).
We apply this construction to differentiable stacks that are smooth orbifolds, that is, they can be presented by proper \'etale Lie groupoids (see Definition \ref{def:proper_etale}).

Existing work on mapping spaces of orbifolds has been considered in the purely topological case \cite{CPRST_15,GH_06,Haefliger,Pronk-Scull}, 
the case of $C^k$ maps in \cite{Chen_06}, Sobolev maps in \cite{Weinmann} and smooth maps in \cite{Borzellino_Brunsden}. In all these cases, some sort of orbifold structure has been found (for instance, Banach or Fr\'echet orbifolds). However \cite{Schmeding_15} gives counterexamples showing that some of the constructions in \cite{Borzellino_Brunsden} are not well-defined.

Noohi \cite{Noohi_10} solved the problem of constructing a topological mapping stack between more general \emph{topological} stacks, when the source stack has a presentation by a compact topological groupoid.
See \cite{Roberts-Vozzo} for further references and discussion.

We take as given the definition of Lie groupoid in what follows, using finite-dimensional manifolds unless otherwise specified. 
Manifolds will be considered as trivial groupoids without comment.
We pause only to note that in the infinite-dimensional setting, the source and target maps of Fr\'echet--Lie groupoids must be submersions between Fr\'echet manifolds, which is a stronger hypothesis than asking the derivative is surjective (or even split) everywhere, as in the finite-dimensional or Banach case. 

We will also consider groupoids in diffeological spaces. 
Diffeological spaces (see for instance \cite{Baez-Hoffnung_11}) contain Fr\'echet manifolds as a full subcategory and admit all pullbacks (in fact all finite limits) and form a cartesian closed category such that for $K$ and $M$ smooth manifolds with $K$ compact, the diffeological mapping space $M^K$ is isomorphic to the Fr\'echet manifold of smooth maps $K \to M$.

Differentiable stacks are, for us, stacks of groupoids on the site $\cM$ of finite-dimensional smooth manifolds with the open cover topology that admit a presentation by a Lie groupoid \cite{Behrend-Xu_11}. 
We can also consider the more general notion of stacks that admit a presentation by a diffeological or Fr\'echet--Lie groupoid.

\begin{definition}
	
	Let $\cX,\cY$ be stacks on $\cM$. The \emph{Hom-stack} $\cHom(\cY,\cX)$ is defined by taking the value on the manifold $N$ to be $\Stack_\cM(\cY\times N,\cX)$.

\end{definition}

Thus we have a Hom-stack for any pair of stacks on $\mathcal{M}$.
The case we are interested in is where we have a differentiable stack $\cX$ associated to a Lie groupoid $X$, e.g.~an orbifold, and the resulting Hom-stack $\cHom(M,\cX)$ for $M$ a compact manifold.

We define a \emph{minimal cover} of a manifold $M$ to be a cover by regular closed sets $V_i$ such that the interiors $\interior{V_i}$ form an open cover of $M$, and every $\interior{V_i}$ contains a point not in any other $\interior{V_j}$. 
We also ask that finite intersections $V_i\cap\ldots\cap V_k$ are also regular closed.
Denote the collection of minimal covers of a manifold $M$ by $C(M)_{\min}$, and note that such covers are cofinal in open covers.  
Recall that a cover $V$ of a manifold defines a diffeological groupoid $\vC(V)$ with objects $\coprod_i V_i$ and arrows $\coprod_{i,j} V_i \cap V_j$.\footnote{We can in what follows safely ignore the issue of intersections of boundaries.} 
We are particularly interested in the case when we take the closure $\{\overline{U_i}\}$ of $\{U_i\}$, a good open cover, minimal in the above sense.

We denote the arrow groupoid of a Lie groupoid $X$ by $X^\mathbbm{2}$---it is again a Lie groupoid and comes with functors $S,T\colon X^\mathbbm{2} \to X$, with object components source and target, resp. Let $M$ be a compact manifold with corners and $X$ a Lie groupoid. 
Define the \emph{mapping groupoid} $\Map(M, X)$ to be the following diffeological groupoid. The object space $\Map(M, X)_0$ is the disjoint union over minimal covers $V$ of the spaces $\inhom{\vC(V)}{X}$ of functors $\vC(V) \to X$. 
The arrow space $\Map(M, X)_1$ is
\[
	\coprod_{V_1,V_2\in C(M)_\text{min}} \inhom{\cech{V_1}}{X} \times_{ \inhom{\cech{V_{12}}}{X}} 
	\inhom{\cech{V_{12}}}{(X^\mathbbm{2})}
	\times_{ \inhom{\cech{V_{12}}}{X}} \inhom{\cech{V_2}}{X} 
\]
where the chosen minimal refinement $V_{12}\subset V_1 \times_M V_2$ is defined using the boolean product on the algebra of regular closed sets. 
The maps
\begin{equation}\label{eq:mapping_gpd_construction_maps}
	S,T\colon \inhom{\cech{V_{12}}}{(X^\mathbbm{2})} \to \inhom{\cech{V_{12}}}{X}
	\qquad\text{and}\qquad
	 \inhom{\cech{V_i}}{X}\to  \inhom{\cech{V_{12}}}{X} \quad (i=1,2)
\end{equation}
give us a pullback and the two projections
\begin{equation}\label{eq:mapping_gpd_source_target}
	\inhom{\cech{V_1}}{X} \times_{ \inhom{\cech{V_{12}}}{X}} 
	\inhom{\cech{V_{12}}}{(X^\mathbbm{2})}
	\times_{ \inhom{\cech{V_{12}}}{X}} \inhom{\cech{V_2}}{X} 
	\longrightarrow \inhom{\cech{V_i}}{X},
\end{equation}
induce, for $i=1,2$, the source and target maps for our groupoid resp.
Composition in the groupoid is subtle, but is an adaptation of the composition of transformations of anafunctors given in \cite{Roberts_12}.
The proof of the following theorem works exactly as in Theorem 4.2 in \cite{Roberts-Vozzo}.

\begin{theorem}

	For $X$ a Lie groupoid and $M$ a compact manifold the Hom-stack $\cHom(M,X)$ is presented by the diffeological groupoid $\Map(M,X)$. \qed

\end{theorem}

We need some results that ensure the above constructions give Fr\'echet manifolds.

\begin{proposition}\label{restriction_is_subm}
	For $M$ a compact smooth Riemannian manifold (possibly with corners), $K$ a compact regular closed Lipschitz subset of $M$ and $N$ a smooth manifold, the induced restriction map $N^M \to N^K$ is a submersion of Fr\'echet manifolds.
\end{proposition}

\begin{proof}
Recall that a submersion of Fr\'echet manifolds is a smooth map that is locally, for suitable choices of charts, a projection out of a direct summand. 
This means we have to work locally in charts and show that we have a split surjection of Fr\'echet spaces. 
We can reduce this to the case that $N=\mathbb{R}^n$, since the charts are given by spaces of sections of certain vector bundles, and we can consider these spaces locally and patch them together, and thence to $N=\mathbb{R}$. 
The proof then uses \cite[Theorem 3.15]{Frerick}, as we can work in charts bi-Lipschitz to flat $\mathbb{R}^n$, hence reduce to the case of $K\subset \overline{B}\subset \mathbb{R}^n$, for $B$ some large open ball. \end{proof}

In particular this is true for sets $K$ that are closures of open geodesically convex sets, and even more specifically such open sets that are the finite intersections of geodesically convex charts in a good open cover.
We also use a special case of Stacey's theorem \cite[Corollary 5.2]{Stacey_13}; smooth manifolds with corners are smoothly $\mathfrak{T}$-compact spaces in Stacey's sense. Another proof, not using the technology of generalised smooth spaces, is given in \cite[Lemma 2.4]{Amiri-Schmeding}.

\begin{theorem}[Stacey]\label{Staceys_thm}
	Let $N_1 \to N_2$ be a submersion of finite-dimensional manifolds and $K$ a compact manifold, possibly with corners.
	Then the induced map of Fr\'echet manifolds $N_1^K \to N_2^K$ is a submersion. \qed
\end{theorem}

The following proposition is the main technical tool in proving the mapping stack is an infinite-dimensional differentiable stack.

\begin{proposition}\label{prop:space_of_functors_is_mfld}

	The diffeological space $\inhom{\vC(V)}{X}$ is a Fr\'echet manifold.

\end{proposition}

\begin{proof}
First, the diffeological space of functors is isomorphic to the space of simplicial maps $N\vC(V) \to NX$ between the nerves of the groupoids. 
Then, since the subspaces of degenerate simplicies in $N\vC(V)$ are disjoint summands, we can remove those, and consider semi-simplicial maps between semisimplicial diffeological spaces instead. 
Then, since inverses in $\vC(V)$ are also disjoint, we can remove those as well, and consider the diffeological space of semisimplicial maps from the `nerve' of the smooth irreflexive partial order $\vC^<(V)$ to the nerve of $X$, considered as a semisimplicial space (where we have chosen an arbitrary total ordering on the finite minimal cover $V$). 
\emph{This} diffeological space is what we show is a Fr\'echet manifold, by carefully writing the limit as an iterated pullback of diagrams involving maps that are guaranteed to be submersions by Proposition \ref{restriction_is_subm} and Theorem \ref{Staceys_thm}, and using the fact that $X$ is appropriately coskeletal, i.e.~$(NX)_n = X_1 \times_{X_0}\ldots \times_{X_0} X_1$. 
The original space of functors is then a diffeological space isomorphic to this Fr\'echet manifold, hence is a Fr\'echet manifold.
\end{proof}

\begin{lemma}\label{lemma:induced_maps_between_hom_spaces_submersions}
	Let $X \to Y$ be a functor between Lie groupoids with object and arrow components submersions, and $V_1 \to V_2$ a refinement of minimal covers.
	Then the induced map $\inhom{\cech{V_2}}{X} \to \inhom{\cech{V_1}}{Y}$ is a submersion between Fr\'echet manifolds. \qed
\end{lemma}

We will consider the special cases that the functor $X\to Y$ is the identity, and also that the refinement $V_1\to V_2$ is the identity.

\begin{theorem}
	For a Lie groupoid $X$ and compact manifold $M$, $\Map(M, X)$ is a Fr\'echet--Lie groupoid.
\end{theorem}

\begin{proof}
The object space $\Map(M, X)_0$ is a manifold by Proposition \ref{prop:space_of_functors_is_mfld}.
The arrow space $\Map(M,X)_1$ is a manifold since it is given by a pullback diagram built with the maps (\ref{eq:mapping_gpd_construction_maps}), which are submersions by Lemma \ref{lemma:induced_maps_between_hom_spaces_submersions}. 
The identity map is smooth, as it is a smooth map between diffeological spaces that happen to be manifolds, and so is composition. 
\end{proof}

The following theorem is the first main result of the paper. The proof uses the technique of \cite[Theorem 4.2]{Noohi_10} as adapted in \cite{Roberts-Vozzo}, where it is shown that all of the constructions remain smooth.

\begin{theorem}

	For a Lie groupoid $X$ and compact manifold $M$, the stack $\cHom(M,\cX)$ is weakly presented by the Fr\'echet--Lie groupoid $\Map(M,X)$. \qed

\end{theorem}

A \emph{weak} presentation means that the pullback of the map $\Map(M,X)_0 \to \cHom(M,\cX)$ against itself gives a stack representable by the Fr\'echet manifold $\Map(M,X)_1$, and the two projections are submersions. 
For the site of diffeological spaces, a weak presentation is an ordinary presentation. 
This is also the case if we allow non-Hausdorff manifolds \cite[Proposition 2.2]{Behrend-Xu_11}, so we either have to pay the price of a weak presentation or working over a site of non-Hausdorff manifolds. 
If the groupoid $\Map(M,X)$ is \emph{proper}, as in Theorem \ref{thm:mapping_orbifold} below, then we can upgrade this weak presentation to an ordinary one over Hausdorff manifolds.

\begin{definition}\label{def:proper_etale}
	A (Fr\'echet--)Lie groupoid $Z$ is \emph{proper} if the map $(s,t)\colon Z_1 \to Z_0\times Z_0$ is a proper map (i.e.~closed with compact fibres), \emph{\'etale} if the source and target maps are local diffeomorphisms, and an \emph{orbifold groupoid} if it is a proper and \'etale.
\end{definition}

It is a theorem of Moerdijk--Pronk that \emph{effective} orbifold groupoids are equivalent to \emph{reduced} orbifolds defined in terms of orbifold charts---in finite dimensions (see e.g.~\cite{Pohl}). 
For infinite dimensional orbifolds this is not yet known, but may be possible for Fr\'echet--Lie groupoids with local additions on their object and arrow manifolds,\footnote{Alexander Schmeding, Private communication.} of which our construction is an example.

Our second main result is then:

\begin{theorem}\label{thm:mapping_orbifold}

	If $X$ is an \'etale Lie groupoid, then $\Map(M,X)$ is \'etale.
	If $X$ is an orbifold groupoid, then $\Map(M,X)$ is an orbifold groupoid.

\end{theorem}

\begin{proof} Stability of local diffeomorphisms under pullback mean that we only need to show that the smooth maps $\inhom{\cech{V}}{S},\inhom{\cech{V}}{T}\colon\inhom{\cech{V}}{(X^\mathbbm{2})} \to \inhom{\cech{V}}{X}$, for any minimal cover $V$, are local diffeomorphisms.
If $X$ is an \'etale Lie groupoid then the fibres of its source and target maps are discrete, and one can show that $\inhom{\cech{V}}{S},\inhom{\cech{V}}{T}$ have discrete diffeological spaces as fibres. 
But these maps are submersions of Fr\'echet manifolds, hence are local diffeomorphisms.

Properness follows if we can show that $(s,t)$ for the mapping groupoid is closed and every object has a finite automorphism group. 
This reduces to showing that $\inhom{\cech{V}}{(S,T)}$ is closed and its fibres are finite.
We can show the latter by again working in the diffeological category and showing that the fibres of $\inhom{\cech{V}}{(S,T)}$ are discrete, and also a subspace of a finite diffeological space. 
As all the spaces involved are metrisable, we use a sequential characterisation of closedness together with the local structure of the proper \'etale groupoid $X$, and find an appropriate convergent subsequence in the required space of natural transformations. 
\end{proof}


\newcommand{\etalchar}[1]{$^{#1}$}
\providecommand{\bysame}{\leavevmode\hbox to3em{\hrulefill}\thinspace}
\providecommand{\MR}{\relax\ifhmode\unskip\space\fi MR }
\providecommand{\MRhref}[2]{%
  \href{http://www.ams.org/mathscinet-getitem?mr=#1}{#2}
}
\providecommand{\href}[2]{#2}

\end{document}